\documentclass[draft,twoside,9pt,leqno]{amsart}
\usepackage{srcltx}
\usepackage[english]{babel}
\numberwithin{equation}{section}
\newtheorem{thm}{Theorem}[section]

\newtheorem*{thmA}{Theorem A}
\newtheorem*{corB}{Corollary B}

\newtheorem*{thmA'}{Theorem A'}
\newtheorem*{thmB'}{Theorem B'}

\newtheorem*{thmA''}{Theorem A''}
\newtheorem*{thmB''}{Theorem B''}
\newtheorem{lem}[thm]{Lemma}
\newtheorem{pro}[thm]{Proposition}

\theoremstyle{definition}
\newtheorem{defi}[thm]{Definition}

\newtheorem{rmk}[thm]{Remark}

\numberwithin{equation}{section}

\usepackage{amssymb,amsmath}
\usepackage{mathtools}

\def\div{\mathop\mathrm{div}\nolimits}

\def\Lip{\mathop\mathrm{Lip}\nolimits}
\def\tr{\mathop\mathrm{tr}\nolimits}

\def\inn{\mathop\mathrm{int}\nolimits}

\newcommand{\supp}{\operatorname{supp}}
\newcommand{\dist}{\operatorname{dist}}

\newcommand{\R}{\mbox{${\mathbb R}$}}
\newcommand{\N}{\mathbb{N}}

\newcommand{\g}[2]{\langle #1 ,#2 \rangle}

\newcommand{\ep}{\varepsilon}
\newcommand{\rf}[1]{\mbox{(\ref{#1})}}

\newcommand{\ds}[1]{\displaystyle{#1}}

\newcommand{\eig}{\lambda_1^L}
\newcommand{\eigz}{\zeta_1^L}

\newcommand{\eigb}{\lambda_1^{L}}
\newcommand{\M}{(M,\g{\ }{\ })}
\newcommand{\dM}{(M,\partial M,\g{\ }{\ })}

\def\beq{\begin{equation}}
\def\eeq{\end{equation}}
\def\beqs{\begin{equation*}}
\def\eeqs{\end{equation*}}

\begin{document}

\title[Lichnerowicz-type equations on complete manifolds]{Lichnerowicz-type equations \\ with sign-changing nonlinearities \\ on complete manifolds with boundary}

\author[Guglielmo Albanese]{Guglielmo Albanese}
\author[Marco Rigoli]{Marco Rigoli}

\address{Dipartimento di Matematica, Universit\`{a} degli Studi di Milano, Via Saldini 50, I-20133, Milano, Italy}
\email {guglielmo.albanese@carloalberto.org, gu.albanese@gmail.com}
\email {marco.rigoli@unimi.it}


\date{\today}

\begin{abstract}
We prove an existence theorem for positive solutions to Lichnerowicz-type equations on complete manifolds with boundary $\dM$ and nonlinear Neumann conditions. This kind of nonlinear problems arise quite naturally in the study of solutions for the Einstein-scalar field equations of General Relativity in the framework of the so called Conformal Method.
\end{abstract}

\maketitle

\section{Introduction}
In this paper we continue the investigation programme started in \cite{AR} on the behaviour of positive solutions of Lichnerowicz-type equations on complete Riemannian manifolds. While in the aforementioned paper we concentrated on the case of complete manifolds without boundary $\M$, now we turn our attention to the case of complete manifolds $\dM$ with a nonempty boundary $\partial M$ (possibly noncompact). More precisely, we are interested in the problem
		\beq\label{lich}
		\begin{cases}
		\Delta u+a(x)u-b(x)u^{\sigma}+c(x)u^{\tau}= 0 & \hbox{on $\inn M$}\\
		\partial_{\nu}u-g(x,u)= 0 & \hbox{on $\partial M$},
		\end{cases}
		\eeq
where $\Delta$ is the Laplace-Beltrami operator, $\partial_{\nu}$ is the outer normal derivative, 
$\tau<1<\sigma$ are constants, the coefficients $a(x)$, $b(x)$, and $c(x)$ are in $\mathrm{C}^0(M)$, and $g(x,t)\in\mathrm{C}^0(\partial M\times\R^+)$ is a nonlinear function. Equation \rf{lich} is a particular case of a semilinear elliptic equation with a nonlinear Neumann boundary condition, that is,
		\beqs\label{nonlin}
		\begin{cases}
		\Delta u+f(x,u)= 0 & \hbox{on $\inn M$}\\
		\partial_{\nu}u-g(x,u)= 0 & \hbox{on $\partial M$},
		\end{cases}
		\eeqs
where $f(x,t)\in\mathrm{C}^0(M\times\R^+)$ and $g(x,t)\in\mathrm{C}^0(\partial M\times\R^+)$ model the nonlinearities. This kind of problems arise quite naturally in many branches of mathematics, for instance in studying conformal deformations of Riemannian manifolds with boundary \cite{Ch, Es, Esann}, finding optimal constants for Sobolev trace embeddings \cite{Ch, FdP, PT}, and reaction-diffusion equations \cite{BDL, dNdM}.\\

In particular, equations of the form \rf{lich} appear in the context of mathematical general relativity. Indeed, a fruitful way to look for solutions of the Einstein field equations consists in exploiting the conformal method introduced by Lichnerowicz \cite{Li} and York \cite{Yo1, Yo2}. This means that we generate an initial data set $\left(M,\widehat{g},\widehat{K},\widehat{\rho},\widehat{J}\right)$ satisfying the Einstein constraint equations by first choosing the following conformal data: 
	\begin{itemize}
	\item an $m$-dimensional Riemannian manifold $\M$; 
	\item a symmetric $2$-covariant tensor $\sigma$ required to be traceless and transverse with respect to $\g{\,}{\,}$, that is, for which $\tr_{\g{\,}{\,}}\sigma=0$ and $\div_{\g{\,}{\,}}\sigma=0$; 
	\item a scalar function $\tau$. 
	\end{itemize}
Then one looks for a positive function $u$ and a vector field $W$ that solve the conformal constraint system, that, in the case of an Einstein-scalar field (see \cite{CBIP, CB2, NX}) reads as
	\beq\label{confconI}
	\begin{cases}
	\frac{4(m-1)}{m-2}\Delta u-\left(S-\left|\nabla\psi\right|^2\right)u+\left(\left|\sigma+\mathring{\mathcal{L}}W\right|^{2}+\pi^2\right)u^{-N-1}\\ \quad\quad\quad\quad-\left(\frac{m-1}{m}\tau^2-2U(\psi)\right)u^{N-1}=0&\\
	\Delta_{\mathbb{L}}W+\frac{m-1}{m}u^N\nabla\tau-\pi\nabla\psi=0\,,&
	\end{cases}
	\eeq
where $\psi$ is the restriction on $M$ of a scalar field $\Psi$ defined in the whole spacetime, $U(t)$ its potential, and $\pi$ the restriction of its normal time derivative on $M$. Here $\Delta$, $S$, and $\left|\cdot\right|$ denote respectively the Laplace-Beltrami operator, the scalar curvature, and the norm in the metric $\g{\,}{\,}$. The operator $\mathring{\mathcal{L}}$ is the traceless Lie derivative and $\Delta_{\mathbb{L}}=\div\circ\mathring{\mathcal{L}}$ is the vector laplacian (see \cite{CB2}). The constant $N$ appearing in \rf{confconI} is the critical Sobolev exponent given by
	\beqs
	N=\frac{2m}{m-2}\,,
	\eeqs
for $m\geq3$, in particular we note that $(-N-1)<1<(N-1)$. If $\left(u,W\right)$ is a solution of \rf{confconI}  then a suitable rescaling of functions, fields, and potentials leads to a solution of the Einstein constraint equations (see \cite{CBIP, CB2}), thus in turn to Cauchy data for the Einstein equations. For further informations on the initial value problem for the Einstein equations we refer to the recent surveys \cite{BI, CB2, CGP}, and the references therein.

The scalar equation in \rf{confconI} is called the Lichnerowicz equation; since it is the main source of nonlinearity in the system \rf{confconI}, a good understanding of its solutions is a crucial step toward the resolution of the Einstein equations by the conformal method, see for instance \cite{CB2, CBIP, CM, GN, HM1, HPP, HV, HT1, Is1, Ma, MW, NX, Pr1, Pr1}. Here we generalize many of the results obtained in the above papers, especially relaxing the assumptions on $M$. Indeed, to the best of our knowledge, most of the results about Lichnerowicz equation are obtained on manifolds which are either compact or with \emph{asymptotically simple} ends (euclidean, hyperbolic, cylindrical, periodic,\dots) (see for instance \cite{CD, CM, GS}), cases that respectively physically correspond, to a cosmological solution or an isolated system. In this paper we enlarge significantly the family of admissible Riemannian manifolds; from the point of view of general relativity, this means a wider choice for the geometry of the initial data. We stress the fact that the knowledge of the precise behaviour at infinity (i.e. the \emph{asympotically simple} ends), allows the study of the Lichnerowicz equation in a classical analytical setting by means of suitably weighted Sobolev spaces (see for instance the very recent \cite{DIMM}). In our general case we need to use different techniques to tackle the problem.\\

Another natural issue in this framework, inspired by the classical singularity theorems of Hawking and Penrose, is to understand the behaviour of an initial data set containing event horizons. The approach introduced in \cite{YP} and which has been highly developed in recent years, see for instance \cite{HM, HT1, M1}, consists in excising the regions containing black holes and coherently impose some boundary conditions on the conformal factor. The boundary conditions introduced in the aforementioned references can be gathered into the following type of conditions
	\beq\label{genbound}
	\partial_{\nu}u+g_{H,j}u+g_{\theta ,j}u^{e_j}+g_{\tau ,j}u^{N\slash 2}+g_{w,j}u^{-N\slash 2}=0\quad\quad\quad\hbox{on $\partial_jM$}\,,
	\eeq
for each boundary component $\partial_jM$; where the coefficients $g_{\cdot,j}$ and $e_j$, are related to the physical meaning of that boundary component, see \cite{HT1} for a comprehensive exposition of the problem. This means that on a manifold with boundary $\dM$ the Lichnerowicz equation in \rf{confconI} has to be complemented with the boundary condition \rf{genbound}, that is,

	\beqs
	\begin{cases}
	\frac{4(m-1)}{m-2}\Delta u+A(x)u-B(x)u^{N-1}+C(x)u^{-N-1}=0 & \hbox{on $\inn M$}\,,\\
	\partial_{\nu}u+g_{H,j}u+g_{\theta ,j}u^{e_j}+g_{\tau ,j}u^{N\slash 2}+g_{w,j}u^{-N\slash 2}=0 & \hbox{on $\partial_jM$}\,,
	\end{cases}
	\eeqs
where the coefficients are given by
	\beq\label{ABC}
	\hbox{$A(x)=\left(\left|\nabla\psi\right|^2-S\right)$, $B(x)=\left(\frac{m-1}{m}\tau^2-2U(\psi)\right)$, and $C(x)=\left(\left|\sigma+\mathring{\mathcal{L}}W\right|^{2}+\pi^2\right)$.}
	\eeq

The discussion above motivates the study of positive solutions of the following special case of problem \rf{lich}
	\beqs
		\begin{dcases}
		\Delta u+a(x)u-b(x)u^{\sigma}+c(x)u^{\tau}=0 & \hbox{on $\inn M$,}\\
		 \partial_{\nu}u+\sum_{i=1}^{N}g_i(x)u^{q_i}=0& \hbox{on $\partial M$}; 
		\end{dcases}
	\eeqs
on a general $\dM$ with mild conditions on the coefficients $a(x)$, $b(x)$, $c(x)$, $g_i(x)$, $\sigma$, $\tau$, and $q_i$. It follows from \rf{ABC} that our preferred sign for the coefficient $c(x)$ is $c(x)\geq 0$, while the coefficients $a(x)$ and $b(x)$ have no sign restrictions. We stress the fact that the presence of a sign-changing $b(x)$ constitutes a difficulty in finding positive solutions. Thus, denoting by $b_+(x)$ and $b_-(x)$ respectively the positive and negative part of $b(x)$ (that is, $b(x)=b_+(x)-b_-(x)$), we introduce the function 
	\[
	b_{\theta}(x)=b_+(x)-\theta b_-(x)\,,
	\]
where $\theta\in(0,1]$. The function $b_{\theta}(x)$ is perturbation of $b(x)$ that permits to \emph{modulate} his negative part. Our main result is that, it is possible to find a small enough $\theta^*\in(0,1]$ such that the perturbed equation
	\beqs\label{licht}
		\Delta u+a(x)u-b_{\theta}(x)u^{\sigma}+c(x)u^{\tau}= 0 
	\eeqs 
has a positive solution for each $\theta\in(0,\theta^*]$. We note that this kind of perturbation is coherent with other results appeared in the literature, see for instance \cite{HV} and \cite{NX}. Moreover, recalling \rf{ABC}, in the case of Lichnerowicz equation for the Einstein-scalar field, the coefficient $b(x)$ is given by
	\[
	b(x)=\frac{m-2}{4}\left(\frac{1}{m}\tau^2-\frac{2}{m-1}U(\psi)\right)
	\]
thus, 
	\[
	0\leq \theta b_-(x)=\frac{\theta(m-2)}{2(m-1)}U_+(\psi)
	\]
so that the modulation $b_{\theta}(x)$ can be interpreted phisically as a smallness requirement for the positive part $U_+(\psi)$ of the potential.

Concerning the boundary conditions, the restrictions on $g_i(x)$ and $q_i$ that we have in mind are 
	\[
	\min_{1\leq i\leq N} q_i<1<\max_{1\leq i\leq N} q_i
	\]
and such that 
	\[
	g_i(x)\left(q_i-1\right)\geq 0\quad\quad\hbox{for all $i$;}
	\]
this last condition in the literature of mathematical General Relativity is known as the \emph{defocusing case} and it is meaningful for the applications, see for instance the aforementioned \cite{HT1, HM} and the references therein.\\

\section{The geometric setting and main results}
From now on $\dM$ will denote a smooth, complete Riemannian manifold of dimension $m\geq2$, and smooth boundary $\partial M$. It is worth to spend some words on the notion of completeness for a Riemannian manifold with boundary. Indeed in this case the familiar Hopf-Rinow theorem does not hold, because the presence of the boundary prevents the infinite extendability of geodesics. Thus the completeness of $\dM$ has to be understood in the sense of metric spaces. Here the distance between two points $p,q\in M$ is defined as usual as
	\beqs
	\dist(p,q)=\inf_{\sigma\in\Sigma^1_{p,q}}l(\sigma)
	\eeqs
where $\Sigma^1_{p,q}$ is the set of $\mathrm{C}^1$ paths starting at $p$ and ending at $q$, and $l(\sigma)$ is the lenght of $\sigma$ with respect to the metric $\g{\,}{\,}$. We denote with $B_r(x_0)$ the geodesic ball of radius $r\in\R^+$ and center $x_0\in M$, that is
	\[
	B_r(x_0):=\left\{x\in M : \dist(x,x_0)<r\right\}\,.
	\]
The topology on the manifold will be the relative topology, that is, a basis for the open sets is given by the metric balls $B_r(x_0)$ centered at any point of $M$, regardless of belonging or not to the boundary $\partial M$. To be clear, with this topology, the half ball
	\[
	B^+_r(o)=\left\{(x_1,\dots,x_{n-1},x_n)\in\R^n:\sqrt{\sum_{i=1}^nx_i^2}<r,\,x_n\geq0\right\}
	\]
is open in the manifold with boundary $\left(\R^n_+,\R^{n-1},\g{\,}{\,}_{eucl}\right)$. We will adopt the notation $\inn M=M\setminus\partial M$. Now let $\Omega\subset M$ be a domain. To deal with boundary value problems we need to split the boundary, $\partial\Omega$, in two different parts
	\[
	\partial_0\Omega=\partial\Omega\cap\inn M\,;\quad\quad\partial_1\Omega=\partial\Omega\cap\partial M\,,
	\]
clearly $\partial\Omega=\partial_0\Omega\cup\partial_1\Omega$.\\

Let $a(x)\in\mathrm{C}^0(M)$, here we introduce some spectral properties of the Schr\"odinger operator $L=\Delta+a(x)$ on subsets of $M$. If $\Omega$ is a non-empty open set, the first Zaremba eigenvalue $\eigz(\Omega)$ is variationally characterized as	
	\beq\label{eizvar}
	\eigz(\Omega)=\inf\left\{\int_{\Omega}\left|\nabla\varphi\right|^2-a(x)\varphi^2\,:\,\varphi\in\mathrm{W}_0^{1,2}(\Omega\cup\partial_1\Omega)\,,\int_{\Omega}\varphi^2=1\right\}\,,
	\eeq
where $\mathrm{W}_0^{1,2}(\Omega\cup\partial_1\Omega)$ denotes the closure of $\mathrm{C}_0^{1}(\Omega\cup\partial_1\Omega)$ in $\mathrm{W}^{1,2}(\Omega)$. It is clear that the definition above coincide with that of the usual first Dirichlet eigenvalue $\eig(\Omega)$ if $\partial_1\Omega=\emptyset$, in general, since $\mathrm{W}_0^{1,2}(\Omega)\subset\mathrm{W}_0^{1,2}(\Omega\cup\partial_1\Omega)$, it holds that
	\beq\label{dislz}
	\eigb(\Omega)\geq\eigz(\Omega)\,.
	\eeq 
We note that, for $\Omega\subset M$ such that $\partial\Omega$ is Lipschitz and $\partial_1\Omega$ is at least $\mathrm{C}^2$, the infimum is attained by the unique positive eigenfunction $v$ on $\Omega$ of the so-called Zaremba problem (or mixed Dirichlet-Neumann problem)
	\beq\label{eizar}
	\begin{cases}
	\Delta v+a(x)v+\eigz(\Omega)v=0 & \hbox{on $\Omega$}\\
	v=0 & \hbox{on $\partial_0\Omega$}\\
	\partial_{\nu}v=0 & \hbox{on $\inn\partial_1\Omega$}\,\\
	\left\|v\right\|_{\mathrm{L}^2(\Omega)}=1.
	\end{cases}
	\eeq
To the best of our knowledge, the existence theory for \rf{eizar} seems to be absent in the mathematical literature. A proof of the existence and regularity of solutions of \rf{eizar} is reported in the Appendix of this paper for the interested reader.

We need to extend the definition of $\eigz(\Omega)$ to an arbitrary bounded subset $B$ of $M$. We do this by setting
	\beq\label{eizB}
	\eigz(B)=\sup_{\Omega}\eigz(\Omega)
	\eeq
where the supremum is taken over all open bounded sets $B\subset\Omega\subset M$. This definition is well posed, due to the monotonicity of $\eigz$ with respect to the domain, that is
	\[
	\Omega_1\subseteq\Omega_2\quad\quad\hbox{implies}\quad\quad\eigz(\Omega_1)\geq\eigz(\Omega_2)\,.
	\] 
Observe that, in case $B\subset\subset\inn M$ (that is always true when $\partial M=\emptyset$), $\eigz(B)=\eig(B)$ (see for instance Section 6.12 of \cite{MRS}).\\

Now we can state the main results of the paper. The first is an existence theorem for positive solutions of a perturbed companion of \rf{lich} under general hypoteses on the boundary nonlinearity. In what follows we set $B_0=\left\{x\in M\,:\,b(x)\leq0\right\}$ and $C_0=\left\{x\in M\,:\,c(x)=0\right\}$. 
	\begin{thmA}
	Let $\dM$ be complete and suppose that $a(x)$, $b(x)$, $c(x)\in\mathrm{C}_{loc}^{0,\alpha}(M)$ for some $0<\alpha\leq1$, $c(x)\geq0$. Assume that $B_0$, $C_0$ are compact. Let $\sigma$, $\tau\in\R$ be such that $\tau<1<\sigma$. Let $g(x,t)\in\mathrm{C}^0(\partial M\times\R^+)$ satisfy
		\beq\label{condgall}
		\begin{dcases}
		i)\quad\quad\exists\, \overline{\omega}>0\,:\,\inf_{x\in\partial M}g(x,\omega)\geq0\quad\hbox{for all $0<\omega\leq\overline{\omega}$}\,;\\
		ii)\quad\quad\exists\, \overline{\gamma}>0\,:\,\sup_{x\in\partial M}g(x,\gamma)\leq0\quad\hbox{for all $\gamma\geq\overline{\gamma}$}\,;\\
		iii)\quad\quad\lim_{s\rightarrow 0^+}\frac{g(x,s)}{s}=+\infty\,;\\
		iv)\quad\quad\lim_{t\rightarrow\infty}\frac{g(x,t)}{t}=-\infty\,,
		\end{dcases}		
		\eeq
	and such that $\ds{\frac{g(x,t)}{t}}$ is non-increasing. Suppose
		\beq\label{condsupall}
		\begin{dcases}
		i)\quad\quad\eigz(B_0)>0\,;\\
		ii)\quad\quad\zeta_1^{\overline{L}}(C_0)>0\,;\\
		iii)\quad\quad\limsup_{x\rightarrow\infty}\frac{a_+(x)+c(x)}{b_+(x)}<+\infty\,;\\
		iv)\quad\quad\limsup_{x\rightarrow\infty}\frac{a_-(x)+b_+(x)}{c(x)}<+\infty\,,
		\end{dcases}		
		\eeq
	where $L=\Delta+a(x)$ and $\overline{L}=\Delta-a(x)$. Then there exists $\theta_0\in(0,1]$ such that for each $\theta\in(0,\theta_0]$ there exists $u\in \mathrm{C}^2(\inn M)\cap\mathrm{C}^0(M)\cap\mathrm{L}^{\infty}(M)$ positive solution of
	\beq\label{u+-}
	\begin{cases}
	\Delta u+a(x)u-b_{\theta}(x)u^{\sigma}+c(x)u^{\tau}= 0 & \hbox{on $\inn M$}\\
	\partial_{\nu}u-g(x,u)= 0 & \hbox{on $\partial M$},
	\end{cases}
	\eeq
	where $b_{\theta}(x)=b_+(x)-\theta b_-(x)$.
	\end{thmA}
	\begin{rmk}
	Conditions \rf{condsupall} $i)$ and $ii)$ can be thought essentially as a smallness condition for the sets $B_0$ and $C_0$, at least in a spectral sense. This was observed for the Yamabe-type equations on manifolds without boundary some time ago and it is for instance briefly reported on page 159 of \cite{MRS}. Of course, this condition also depends on the behaviour of $a(x)$, therefore \emph{small in a spectral sense} does not necessarily mean, for instance, \emph{small in a Lebesgue measure sense}. In the case of empty boundary $\partial M=\emptyset$, this is clear if $a(x) \leq 0$ because $\lambda^{\Delta}_1(M) \geq 0$ in any complete manifold $\M$ with $\partial M=\emptyset$ so that in this case $\lambda^{L}_1(M) \geq 0$ and thus $\lambda^{L}_1(B) > 0$ on any bounded set $B\subset M$. 
	
	In the same vein it should be reported the $\lambda_{\delta}$ eigenvalue introduced in the recent work by Dilts and Maxwell, \cite{DM}.
	\end{rmk}
The second result concerns a class of nonlinearities $g(x,t)$ that we have in mind in view of applications.  We consider the following \emph{simple} nonlinearity
	\beq\label{gsum}
	g(x,t)=\sum_{i=1}^{\nu}g_i(x)t^{q_i}
	\eeq
where $g_i\in\mathrm{C}^0(\partial M)$ and $q_i\in\R$ are such that $q_1<q_2<\dots<q_{\nu}$. We say that $g(x,t)$ in \rf{gsum} is a \emph{strongly defocusing nonlinearity} if the following conditions are fullfilled
	\beq\label{defoc}
	\begin{dcases}
	i)\quad\quad \hbox{$(q_i-1)g_i(x)\leq 0$ for all $1\leq i\leq \nu$;}\\\\
	ii)\quad\quad 
	\begin{aligned}
	& \hbox{$\exists\,k$ s.t. $q_{k}>1$, $g_{k}(x)\neq 0$, on $\partial M$}\\ 
	& \hbox{and $\frac{g_i(x)}{g_k(x)}\in\mathrm{L}^{\infty}(\partial M)$ for all $i$ s.t. $q_i\leq 1$;}
	\end{aligned}\\\\
	iii)\quad\quad 
	\begin{aligned}
	& \hbox{$\exists\,h$ s.t. $q_{h}<1$, $g_{h}(x)\neq 0$, on $\partial M$}\\ 
	& \hbox{and $\frac{g_i(x)}{g_h(x)}\in\mathrm{L}^{\infty}(\partial M)$ for all $i$ s.t. $q_i\geq 1$.}
	\end{aligned}
	\end{dcases}
	\eeq
In this case Theorem A yields the following consequence:
	\begin{corB}
	Let $\dM$ be complete and suppose that $a(x)$, $b(x)$, $c(x)$, $\sigma$, and $\tau$ are as in Theorem A. Let $g(x,t)$ be as in \rf{gsum}, satisfying conditions \rf{defoc}. Furthermore assume that the conditions in \rf{condsupall} are satisfied.
	Then there exists $\theta_0\in(0,1]$ such that for each $\theta\in(0,\theta_0]$ there exists $u\in \mathrm{C}^2(\inn M)\cap\mathrm{C}^0(M)\cap\mathrm{L}^{\infty}(M)$ positive solution of \rf{u+-}.
	\end{corB}

\section{The method of sub/super solutions}\label{MIS}
In this section we provide a generalization of the classical \emph{method of sub/supersolutions} (see for instance \cite{Sa}) taylored to find positive solutions of semilinear equations on a complete manifold with boundary $\dM$, with a possibly nonlinear boundary condition, that is
	\beq\label{equ}
	\begin{cases}
	\Delta u+f(x,u)= 0 & \hbox{on $\inn M$}\\
	\partial_{\nu}u-g(x,u)= 0 & \hbox{on $\partial M$}.
	\end{cases}
	\eeq
The global solutions will be obtained as limits of solutions defined on subsets of $M$, thus we need to introduce an adequate family of subsets, that at the same time has to be an exhaustion of $\inn M$ and of $\partial M$. We set the following
	\begin{defi}\label{dregex}
	Let $\dM$ be a complete manifold with boundary and $\left\{\Omega_n\right\}_{n\in\N}$ a family of relatively compact open subsets of $M$ with smooth boundary. Then we say that $\left\{\Omega_n\right\}_{n\in\N}$ is a \emph{$\partial$-regular} exhaustion of $M$ if it satisfies the following conditions:
		\begin{itemize}
		\item $\Omega_n\subset\subset\Omega_{n+1}$ for all $n\in\N$, and $\Omega_n\nearrow M$;
		\item $\partial_1\Omega_n\subset\subset\partial_1\Omega_{n+1}$ for all $n\in\N$, and $\partial_1\Omega_n\nearrow \partial M$;\\
		\end{itemize}
	\end{defi}
	\begin{rmk}
	On a complete manifold $\dM$ there always exists a $\partial$-regular exhaustion. Indeed, fix an origin $o\in M$ and define for each $n\in\N$
		\[
		B_n=B_n(o)\,,
		\]
	since $\dM$ is complete as a metric space, $B_n$ is an exhaustion satisfying the conditions of the definition above. The problem is that $\partial B_n$ need not to be smooth, thus we need to \emph{regularize} it. We claim that for all $n\in\N$ there exists a set $\Omega_n$ such that $B_n\subset\Omega_n\subset\subset B_{n+1}$ and $\partial\Omega_n$ is smooth. This follows from the approximation theory of Lipschitz submanifolds by smooth ones (see for instance \cite{Hi}) since $B_n\subset\subset B_{n+1}$, and $\partial B_n\in\Lip$. 
	\end{rmk}
To solve such nonlinear boundary value problems, we shall need a generalization of the monotone iteration scheme (see for instance \cite{Am} or \cite{Sa}) for semilinear elliptic equations with nonlinear boundary conditions. It is a generalization of Theorem 6.19 of \cite{MRS}. Let $\Omega$ be a bounded open domain, and let $f(x,s)\in\mathrm{C}^0(\Omega\times\R)$,  $g(x,s)\in\mathrm{C}^0(\partial\Omega\times\R)$. Assume that $\beta(x)\in\mathrm{C}^{1,\alpha}(\partial\Omega)$ and define the boundary operator $B$ acting on elements of $\mathrm{C}^1(\overline{\Omega})$ as
	\[
	Bu=\partial_{\nu}u+\beta(x)u\,.
	\]
Then $u_+\in\mathrm{C}^1(\overline{\Omega})$ is a supersolution of 
	\beq\label{bvpmis}
	\begin{cases}
	\Delta u+f(x,u)=0 & \hbox{in $\Omega$}\,,\\
	Bu=g(x,u)&\hbox{on $\partial\Omega$}\,.
	\end{cases}
	\eeq
if		
	\[
	\begin{cases}
	\Delta u_++f(x,u_+)\leq 0 & \hbox{in $\Omega$}\,,\\
	Bu_+\geq g(x,u_+)&\hbox{on $\partial\Omega$}\,.
	\end{cases}
	\]
where the first differential inequality has to be understood in the weak sense. The definition of a subsolution is obtained reversing the inequalities.
	\begin{thm}\label{satt}
	Let $\Omega$ be a relatively compact open domain with smooth boundary $\partial \Omega$. Let $f:\overline{\Omega}\times\R\rightarrow\R $ be a locally H\"older function such that $s\rightarrow f(x,s)$ is locally Lipschitz with respect to $s$ uniformly with respect to $x$ and $g(x,s)\in\mathrm{C}^{2,\alpha}(\partial\Omega\times\R)$. Suppose that $\varphi$ and $\psi\in\mathrm{C}^{1}(\overline{\Omega})$ are respectively a subsolution and a supersolution of \rf{bvpmis} satisfying
		\[
		\varphi\leq\psi\quad\hbox{on $\overline{\Omega}$}\,.
		\]
	Then \rf{bvpmis} has a solution $u\in\mathrm{C}^{2,\alpha}(\overline{\Omega})$ satisfying $\varphi\leq u\leq\psi$.
		
	\end{thm}
		\begin{proof}
		Since $\varphi$ and $\psi$ are bounded, say $\varphi(\overline{\Omega}), \psi(\overline{\Omega})\subseteq[a,b]\subset\R$, there exist positive constants $H$ and $K$ such that the functions
			\[
			s\rightarrow f(x,s)+Hs=F(x,s)\,,
			\]
		and
			\[
			s\rightarrow g(x,s)+Ks=G(x,s)\,,
			\]
		are monotone increasing in $s\in[a,b]$ for every fixed $x$. For each $w\in\mathrm{C}^{\alpha}(\overline{\Omega})$ we let $v=Tw\in\mathrm{C}^{2,\alpha}(\overline{\Omega})$ be the solution of the boundary value problem
			\beq
			\begin{cases}
			(\Delta-H)v=-F(x,w) & \hbox{in $\Omega$}\,,\\
			(B+K)v=G(x,w)&\hbox{on $\partial\Omega$}\,.
			\end{cases}
			\eeq
		which exists and is unique by classical elliptic theory (see for instance \cite{GT}, Theorem 6.30). By the monotonicity of $F(x,s)$ and $G(x,s)$ it follows that the operator $T$ is monotone, that is, if $w_1\leq w_2$ on $\Omega$ then $Tw_1\leq Tw_2$. Indeed the function $\widetilde{v}=Tw_2-Tw_1$ satisfies
			\[
			\begin{cases}
			(\Delta-H)\widetilde{v}=-\left[F(x,w_2)-F(x,w_1)\right]\leq0 & \hbox{in $\Omega$}\,,\\
			(B+K)\widetilde{v}=\left[G(x,w_2)-G(x,w_1)\right]\geq 0&\hbox{on $\partial\Omega$}\,.
			\end{cases}
			\]
		and therefore, by the strong maximum principle $\widetilde{v}\geq 0$ on $\Omega$.
		
		Now we set $u_1^-=T\varphi$, $u_1^+=T\psi$, and for every $k\geq1$, $u_{k+1}^{\pm}=Tu_k^{\pm}$. Reasoning inductively as above we obtain
			\[
			\varphi\leq u_1^-\leq u_2^-\leq\dots\leq u_k^-\leq\dots\leq u_k^+\leq\dots\leq u_2^+\leq u_1^+\leq\psi\,.
			\]
		Thus there exist $u^-$ and $u^+$ such that $u_k^{\pm}\rightarrow u^{\pm}$. The regularity of $u^{\pm}$ and the fact that they are solutions of \rf{bvpmis} follow as in the proof of Theorem 6.19 of \cite{MRS}.
		\end{proof}
Thus we have the following result
		\begin{pro}\label{perron}
		Let $f:M\times\R\rightarrow\R $ be a locally H\"older function such that $s\rightarrow f(x,s)$ is locally Lipschitz with respect to $s$ uniformly with respect to $x$ and $g(x,s)\in\mathrm{C}^{2,\alpha}(\partial M\times\R^+)$. Suppose that $u^-$ and $u^+\in\mathrm{C}^{1}(M)$ are respectively a subsolution and a supersolution of \rf{equ} satisfying
		\[
		0\leq u^-<u^+\quad\hbox{on $M$}\,.
		\]
	Then \rf{equ} has a solution $u\in\mathrm{C}^{2}(M)$ satisfying $u^-\leq u\leq u^+$.
		\end{pro}
			\begin{proof}
			Let $\left\{\Omega_n\right\}_{n\in\N}$ be a $\partial$-regular exhaustion of $M$. Then there exists a family $\left\{\Gamma_n\right\}_{n\in\N}$ of relatively compact open subsets of $\partial M$ such that $\Gamma_n\subset\subset\partial_1\Omega_n$ and $\Gamma_n\nearrow\partial M$. Now, consider a family of cutoff functions $\left\{\psi_n\right\}_{n\in\N}$ such that $\psi_n\in\mathrm{C}^{\infty}(\partial\Omega_n)$ and
				\begin{itemize}
				\item $\supp\psi_n\subset\subset\partial_1\Omega_n$;
				\item $0\leq\psi_n\leq1$;
				\item $\psi_n\equiv 1$ on $\Gamma_n$.
				\end{itemize}
			Now, for each $n\in\N$ we introduce the following family of problems
				\beq\label{Pn}
				\begin{cases}
				\Delta w+f(x,w)= 0 & \hbox{on $\inn \Omega_n$}\\
				\partial_{\nu}w-g_n(x,w)= 0 & \hbox{on $\partial \Omega_n$}\,,
				\end{cases}
				\eeq
			here $g_n(x,w)$ is
				\[
				g_n(x,w)=\psi_n(x)g(x,w)+A_n\left(1-\psi_n(x)\right)\left(\frac{m(x)-w}{\delta(x)}\right)
				\]
			where $m(x)=\frac{1}{2}\left(u^++u^-\right)$, $\delta(x)=\frac{1}{2}\left(u^+-u^-\right)>0$, and
				\[
				A_n=\max\left\{\max_{x\in\partial_1\Omega_n}g(x,u^-),\max_{x\in\partial\Omega_n}\partial_{\nu}u^-,\max_{x\in\partial_1\Omega_n}\left(g(x,u^+)\right)_-,\max_{x\in\partial\Omega_n}\left(\partial_{\nu}u^+\right)_-\right\}\,.
				\]
			We claim that $u^+$ is a supersolution of \rf{Pn} for each $n\in\N$; since the first inequality is trivially satisfied we are left to check the boundary condition, that is,
				\beqs
				\begin{aligned}
				\partial_{\nu}u^+-g_n(x,u^+) & =\psi_n(x)\left[\partial_{\nu}u^+-g(x,u^+)\right]+\left(1-\psi_n(x)\right)\left[\partial_{\nu}u^++A_n\right]\\
				& \geq \left(1-\psi_n(x)\right)\left[\partial_{\nu}u^++A_n\right]\\
				& \geq \left(1-\psi_n(x)\right)\left[\left(g(x,u^+)+A_n\right)\chi_{\partial_1\Omega_n}+\left(\partial_{\nu}u^++A_n\right)\chi_{\partial_0\Omega_n}\right]\\
				& \geq 0\,,
				\end{aligned}
				\eeqs
			where, the first and the second inequalities follow from the facts that $\supp\psi_n\subset\subset\partial_1\Omega_n$ (we recall that $\chi_V$ is the characteristic function of the set $V$), and that $u^+$ is a supersolution, the last inequality is a consequence of the definition of $A_n$. Analougusly it can be shown that $u^-$ is a subsolution of \rf{Pn}. From the monotone iteration scheme (see for instance Theorem \ref{satt}) it follows that there exist a family $\left\{u_n\right\}_{n\in\N}$ of positive solutions of the problems \rf{Pn}. By standard elliptic regularity theory (see for instance \cite{GT}, Theorem 6.31) and the fact that $\Gamma_n\subset\subset\Gamma_{n+1}$, it follows that for $m\geq n$, $u_m\in\mathrm{C}^{2,\alpha}(\inn\Omega_n\cup\Gamma_n)$ and are uniformly bounded there, thus, up to a subsequence, $u_m\rightarrow u^*_n\in\mathrm{C}^2(\inn\Omega_n\cup\Gamma_n)$. Now, since $\inn\Omega_n\cup\Gamma_n\nearrow M$, we can arrange a diagonal subsequence such that $u_m\rightarrow u\in\mathrm{C}^2(M)$ positive solution of \rf{equ}.
			\end{proof}

\section{Construction of a supersolution}\label{seclichsup}
In this section we prove that, under suitable spectral assumptions and assuming a control on the coefficients, we can find a positive supersolution of \rf{u+-}. Namely, we have the following

	\begin{pro}\label{thmsup}
	Let $\dM$ be complete, $a(x)$, $b(x)$, $c(x)\in\mathrm{C}_{loc}^{0,\alpha}(M)$ for some $0<\alpha\leq1$, $c(x)\geq0$, and suppose that $B_0$ is compact. Let $\sigma$, $\tau\in\R$ be such that $\tau<1<\sigma$. Let $g(x,t)\in\mathrm{C}^0(\partial M\times\R^+)$ satisfy
		\beq\label{condg}
		\begin{dcases}
		i)\quad\quad\exists\, \overline{\gamma}>0\,:\,\sup_{x\in\partial M}g(x,\gamma)\leq0\quad\hbox{for all $\gamma\geq\overline{\gamma}$}\\
		ii)\quad\quad\lim_{t\rightarrow\infty}\frac{g(x,t)}{t}=-\infty\,.
		\end{dcases}		
		\eeq
	Assume that
		\beq\label{condsup}
		\begin{dcases}
		i)\quad\quad\eigz(B_0)>0\\
		ii)\quad\quad\limsup_{x\rightarrow\infty}\frac{a_+(x)+c(x)}{b_+(x)}<+\infty\,.
		\end{dcases}		
		\eeq
	Then there exists $\theta_0\in(0,1]$ such that for each $\theta\in(0,\theta_0]$ there exists $u\in \mathrm{C}^2(\inn M)\cap\mathrm{C}^0(M)\cap\mathrm{L}^{\infty}(M)$ positive solution of
	\beq\label{u+}
	\begin{cases}
	\Delta u+a(x)u-b_{\theta}(x)u^{\sigma}+c(x)u^{\tau}\leq 0 & \hbox{on $\inn M$}\\
	\partial_{\nu}u-g(x,u)\geq 0 & \hbox{on $\partial M$}.
	\end{cases}
	\eeq
	\end{pro}
	
	\begin{proof}
	Since $\dM$ is a complete metric space, $B_0$ is bounded. By the definition of $\eigz$ and assumption \rf{condsup} i) we can choose two bounded open sets $D'$ and $D$ such that $\partial D\cap\inn M$ is smooth,
		\beqs
		B_0\subset\subset D'\subset\subset D\quad\quad\hbox{and}\quad\quad\eigz(D)>0\,.
		\eeqs
	Let $u_1$ be the positive solution of 
			\beqs
			\begin{cases}
			\Delta u_1+a(x)u_1+\eigz(D)u_1=0 & \hbox{on $D$}\\
			u_1=0 & \hbox{on $\partial_0 D$}\\
			\partial_{\nu} u_1=0 & \hbox{on $\partial_1 D$}
			\,,
			\end{cases}
			\eeqs
		such that $\left\|u_1\right\|_{\mathrm{L}^{\infty}(D)}=1$.	Now we choose a cut-off function $\psi\in\mathrm{C}_0^{\infty}(M)$ such that $0\leq\psi\leq1$, $\psi\equiv1$ on $D'$, and $\supp\psi\subset D$.
		For positive constants $\eta$, $\mu$, we define
			\beq\label{cutoff}
			u=\eta\left(\psi u_1+(1-\psi)\mu\right)\,.
			\eeq
		We claim that it is possible to choose $\eta$, $\mu$ and $\theta_0$ appropriately so that $u$ becomes a positive solution of \rf{u+}. First we analyze the case $\left\|c\right\|_{\mathrm{L}^{\infty}(D')}\neq 0$, the alternative case being simpler will be discussed later.
		First of all we consider the behaviour of $u$ in $\inn M$. We define $\xi=\left(\inf_{D'}u_1\right)^{\tau-1}>0$. Since $\eigz(D)>0$ and using the fact that $\left\|u_1\right\|_{\mathrm{L}^{\infty}(D)}=1$, on $\overline{D'}$ we have
			\beqs
			\begin{aligned}
			Lu-b_{\theta}(x)u^{\sigma}+c(x)u^{\tau} & = L\left(\eta u_1\right)-b_{\theta}(x)\left(\eta u_1\right)^{\sigma}+c(x)\left(\eta u_1\right)^{\tau}\\
			& \leq \left(\eta u_1\right)\left[-\eigz(D)+\theta b_-(x)\left(\eta u_1\right)^{\sigma-1}+c(x)\left(\eta u_1\right)^{\tau-1}\right]\\
			& \leq \left(\eta u_1\right)\left[-\eigz(D)+\theta \left\|b_-\right\|_{\mathrm{L}^{\infty}(M)}\eta^{\sigma-1}+\xi\left\|c\right\|_{\mathrm{L}^{\infty}(D')}\eta^{\tau-1}\right]
			\end{aligned}
			\eeqs
		To study the sign of the RHS of the above inequality  we need the following elementary calculus lemma.
		\begin{lem}\label{pqlem}
		Let $A$, $B$, $C$ be positive constants. For $t\in\R^+$ consider the function
			\[
			f(t)=At^p+Bt^q-C
			\] 
		where $q<0<p$. Define the positive constant
			\[
			M(p,q)=\left(-\frac{q}{p}\right)^{p\slash(p-q)}+\left(-\frac{q}{p}\right)^{q\slash(p-q)}\,.
			\]
		If 
			\beq\label{suflem}
			A^{-q}B^p<\left(\frac{C}{M(p,q)}\right)^{p-q}
			\eeq
		then $f(t)$ attains a negative minimum at the point
			\beq\label{minlem}
			\overline{t}=\left(-\frac{qB}{pA}\right)^{1\slash(p-q)}.
			\eeq
		\end{lem}
		\begin{proof}[Proof of Lemma \ref{pqlem}]
		Since $\displaystyle{\lim_{t\rightarrow 0^+}f(t)=\lim_{t\rightarrow+\infty}f(t)=+\infty}$, it follows that $f(t)$ attains an absolute minimum at some point $\overline{t}\in\R^+$ such that $f'(\overline{t})=0$, where
			\[
			f'(t)=pAt^{p-1}+qBt^{q-1}\,.
			\]
		It follows that
			\[
			\overline{t}=\left(-\frac{qB}{pA}\right)^{1\slash(p-q)}\,.
			\]
		Condition \rf{suflem} guarantees that $f(\overline{t})<0$.
		\end{proof}
		Going back to the proof of Proposition \ref{thmsup} we set $A=\theta \left\|b_-\right\|_{\mathrm{L}^{\infty}(M)}$, $B=\xi\left\|c\right\|_{\mathrm{L}^{\infty}(D')}$, $C=\eigz(D)$, $p=\sigma-1$, and $q=\tau-1$. In this case, \rf{suflem} and \rf{minlem} read as functions of $\theta$ respectively as
			\[
			\theta<M_1
			\]
		and
			\[
			\overline{t}=\overline{t}(\theta)=\theta^{1\slash(\tau-\sigma)}M_2\,,
			\]
		where $M_1$ and $M_2$ are positive constants depending only on $\left(b(x),c(x),D, D', \sigma, \tau\right)$. We note that $\overline{t}(\theta)\nearrow+\infty$ as $\theta\downarrow0$, thus, there exists $0<\theta_*<M_1$ such that $\overline{t}(\theta)>1$ for $\theta<\theta_*$. Setting 
			\beq
			\eta=\overline{t}(\theta)
			\eeq 
		in \rf{cutoff}, we deduce
			\beq
			Lu-b_{\theta}(x)u^{\sigma}+c(x)u^{\tau}\leq 0 \quad\quad \hbox{on $\overline{D'}\cap \inn M$}\,,
			\eeq
		for each $\theta\leq\theta_*$.
		
		If $\left\|c\right\|_{\mathrm{L}^{\infty}(D')}=0$ we proceed as above, inferring that on $\overline{D}$
			\beqs
			\begin{aligned}
			Lu-b_{\theta}(x)u^{\sigma}+c(x)u^{\tau} & \leq \left(\eta u_1\right)\left[-\eigz(D)+\theta \left\|b_-\right\|_{\mathrm{L}^{\infty}(M)}\eta^{\sigma-1}\right]\,.
			\end{aligned}
			\eeqs
		In this case it is easier to analyze the RHS, indeed it is apparent that it has an absolute minimum at $\eta=0$ with a negative value if
			\[
			\eta<\widetilde{t}(\theta)=\left(\frac{\eigz(D)}{\theta \left\|b_-\right\|_{\mathrm{L}^{\infty}(M)}}\right)^{1\slash (\sigma-1)}\,.
			\] 
		Since we are interested in a positive $u$, we set, to fix ideas,
			\[
			\eta=\frac{1}{2}\widetilde{t}({\theta})\,.
			\] 
		
		Next we consider $M \setminus D$. Since $\supp\psi\subset D$, it follows that $u=\eta\mu$ there. Thus, using \rf{condsup} $ii)$, there exists a $\Gamma>0$ such that, 
		for $\theta<\theta_*$ we have
			\beqs
			\begin{aligned}
			L u-b_{\theta}(x)u^{\sigma}+c(x)u^{\tau} & \leq  a_+(x)\eta\mu-b_+(x)\left(\eta\mu\right)^{\sigma}+c(x)b_+(x)\left(\eta\mu\right)^{\tau}\\
			& \leq b_+(x)\eta\mu\left[\frac{a_+}{b_+}(x)-\left(\eta\mu\right)^{\sigma-1}+\frac{c}{b_+}(x)\left(\eta\mu\right)^{\tau-1}\right]\\
			& \leq b_+(x)\eta\mu\left[\Gamma-\left(\eta\mu\right)^{\sigma-1}+\Gamma\left(\eta\mu\right)^{\tau-1}\right]\\
			\end{aligned}
			\eeqs 	
		and there exists a positive constant $\Lambda_0$ (not depending on $\theta$) such that the RHS is negative for $\mu>\Lambda_0$.\\
		It remains to analyze the situation on $D\setminus\overline{D'}$. First of all we note that, by standard elliptic regularity theory (see \cite{GT}), $u_1\in\mathrm{C}^2(D)$. Thus, since $\supp\psi\subset D$, it follows that $u\in\mathrm{C}^2(M)$, in particular this implies that there exist positive constants $H$, $K$ such that
			\beq\label{HK}
			\begin{cases}
			L u\leq\eta H & \hbox{on $D\setminus\overline{D'}$}\,,\\
			\partial_{\nu}u\geq-\eta K &\hbox{on $\partial_1\left(D\setminus\overline{D'}\right)$}\,,
			\end{cases}
			\eeq
		As a consequence, on $D\setminus\overline{D'}$ we have
			\beqs
			\begin{aligned}
			L u-b_{\theta}(x)u^{\sigma}+c(x)u^{\tau} & \leq \eta H-b(x)\eta^{\sigma}\left(\psi u_1+\left(1-\psi\right)\mu\right)^{\sigma}\\
			& \quad\quad+c(x)\eta^{\tau}\left(\psi u_1+\left(1-\psi\right)\mu\right)^{\tau}\,.
			\end{aligned}
			\eeqs
		Because of our choices of $\psi$ and $D$, there exist positive constants $\ep$ and $E$ such that
			\beqs
			\begin{aligned}
			& \inf_{D\setminus\overline{D'}}b(x)\left(\psi u_1+\left(1-\psi\right)\mu\right)^{\sigma}\geq\ep\,,\\
			& \sup_{D\setminus\overline{D'}}c(x)\left(\psi u_1+\left(1-\psi\right)\mu\right)^{\tau}\leq E\,.
			\end{aligned}
			\eeqs
		Therefore, on $D\setminus\overline{D'}$
			\beqs
			L u-b_{\theta}(x)u^{\sigma}+c(x)u^{\tau}\leq\eta\left(H-\ep\eta^{\sigma-1}+E\eta^{\tau-1}\right)\,.
			\eeqs
		Since $\sigma>1$ and $\tau<1$, we deduce the existence of a constant $\Lambda_1>0$ depending only on $D$ and $D'$ such that
			\beqs
			H-\ep\eta^{\sigma-1}+E\eta^{\tau-1}\leq0
			\eeqs
		for $\eta\geq\Lambda_1$.\\
		
		Now, from \rf{cutoff} and \rf{HK} we infer 
			\[
			\partial_{\nu}u-g(x,u)=-g(x,u)\quad\quad\hbox{on $\partial M\setminus\partial_1\left(D\setminus\overline{D'}\right)$}
			\]
		and
			\[
			\partial_{\nu}u-g(x,u)\geq -g(x,u)-\eta K\quad\quad\hbox{on $\partial_1\left(D\setminus\overline{D'}\right)$}\,.
			\]
		Since $u_1>0$ on $D$, $\supp\psi\subset D$, and $\mu>0$, it holds that
			\[
			\rho=\inf_{x\in\partial M}\left(\psi u_1+(1-\psi)\mu\right)>0
			\]
		thus, from \rf{condg} i) it follows 
			\beq
			g(x,u)\leq 0 \quad\quad\hbox{on $\partial M$}
			\eeq 
		for $\eta\geq\overline{\gamma}\slash\rho$. Moreover
			\[
			-g(x,u)-\eta K\geq -\eta\rho\left[\frac{g(x,u)}{u}+\frac{K}{\rho}\right]\quad\quad\hbox{on $\partial M$}\,,
			\]
		since $\partial_1\left(D\setminus\overline{D'}\right)$ is bounded, $g(x,t)$ is uniformly continous there (as a function of $x$), thus from \rf{condg} ii) it follows that there exists a constant $\Lambda_2>0$ such that the RHS of the inequality above is non-negative for $\eta\geq\Lambda_2$.\\
		
		Since $\overline{t}(\theta)\nearrow+\infty$ monotonically as $\theta\searrow 0^+$, there exists a $0<\theta_0<\theta_*$ with the property that
			\beqs
			\eta=t_*(\theta)\geq\max\left\{1,\Lambda_1,\Lambda_2,\overline{\gamma}\slash\rho\right\}
			\eeqs
		for $\theta\leq\theta_0$. With this choice of $\theta$ and consequently of $\eta$, and with the previous choice of $\mu$, $u$ is the desired solution of \rf{u+}.
		\end{proof}
		\begin{rmk}\label{stac}
		The solution $u$ of \rf{u+} constructed above is bounded below by a positive constant, in other words
		\[
		\ds{u_*=\inf_{x\in M}u>0}\,.
		\]
		\end{rmk}
		\begin{rmk}
		If $B_0\subset\inn M$ (that is $\partial B_0\subset\inn M$), the mixed spectral condition \rf{condsup} i) can be substituted with the usual Dirichlet spectral condition
			\[
			\eigb(B_0)>0\,.
			\]
		\end{rmk}
		
\section{Construction of a subsolution}\label{secsubl}
In this section we produce positive (and bounded) subsolutions to equation \rf{u+-}. The proof is based on two elementary observations. The first is that since $b_+(x)\geq b_{\theta}(x)$ for any $\theta>0$, then any positive subsolution of
	\beq\label{lich+}
		\begin{cases}
		\Delta u+a(x)u-b_{+}(x)u^{\sigma}+c(x)u^{\tau}= 0 & \hbox{on $\inn M$}\\
		\partial_{\nu}u-g(x,u)= 0 & \hbox{on $\partial M$},
		\end{cases}
	\eeq
is also a subsolution for \rf{u+-}.\\
The second is that equation \rf{lich+} has an interesting symmetry property, indeed, let $a(x)$, $b(x)$, $c(x)\in\mathrm{C}_{loc}^{0,\alpha}(M)$ for some $0<\alpha\leq1$, and $\sigma$, $\tau\in\R$ satisfying $\tau<1<\sigma$. Setting 
	\[
	\begin{cases}
	\overline{a}(x)=-a(x)\\
	\overline{b}(x)=c(x)\\
	\overline{c}(x)=b_+(x)\\
	\overline{\tau}=2-\sigma\\
	\overline{\sigma}=2-\tau
	\end{cases}
	\]
it follows that also $\overline{a}(x)$, $\overline{b}(x)$, $\overline{c}(x)\in\mathrm{C}_{loc}^{0,\alpha}(M)$, and $\overline{\tau}<1<\overline{\sigma}$. Now, suppose that $v\in\mathrm{L}^{\infty}(M)$ is a positive supersolution of 
	\beq\label{licho}
		\begin{cases}
		\Delta u+\overline{a}(x)u-\overline{b}(x)u^{\overline{\sigma}}+\overline{c}(x)u^{\overline{\tau}}= 0 & \hbox{on $\inn M$}\\
		\partial_{\nu}u-\overline{g}(x,u)= 0 & \hbox{on $\partial M$},
		\end{cases}
	\eeq
 where
 	\beq\label{go}
 	\overline{g}(x,t)=-t^2g\left(x,\frac{1}{t}\right)\,,
 	\eeq
 then a simple computation shows that $\ds{u^-=\frac{1}{v}}$ is a positive subsolution of \rf{lich+}.\\
 	\begin{pro}\label{thmsub}
	Let $\dM$ be complete $a(x)$, $b(x)$, $c(x)\in\mathrm{C}_{loc}^{0,\alpha}(M)$ for some $0<\alpha\leq1$, $c(x)\geq0$, and suppose that $C_0$ is compact. Let $\sigma$, $\tau\in\R$ satisfy $\tau<1<\sigma$. Let $g(x,t)\in\mathrm{C}^0(\partial M\times\R^+)$ be such that
		\beq\label{condgo}
		\begin{dcases}
		i)\quad\quad\exists\, \overline{\omega}>0\,:\,\inf_{x\in\partial M}g(x,\omega)\geq0\quad\hbox{for all $0<\omega\leq\overline{\omega}$}\\
		ii)\quad\quad\lim_{s\rightarrow 0^+}\frac{g(x,s)}{s}=+\infty\,.
		\end{dcases}		
		\eeq
	Assume
		\beq\label{condsupo}
		\begin{dcases}
		i)\quad\quad\zeta_1^{\overline{L}}(C_0)>0\\
		ii)\quad\quad\limsup_{x\rightarrow\infty}\frac{a_-(x)+b_+(x)}{c(x)}<+\infty\,,
		\end{dcases}		
		\eeq
	where $\overline{L}=\Delta-a(x)$. Then there exists $u\in\mathrm{C}^2(\inn M)\cap\mathrm{C}^0(M)\cap\mathrm{L}^{\infty}(M)$ positive solution of
	\beq\label{u-}
	\begin{cases}
	\Delta u+a(x)u-b_{+}(x)u^{\sigma}+c(x)u^{\tau}\geq 0 & \hbox{on $\inn M$}\\
	\partial_{\nu}u-g(x,u)\leq 0 & \hbox{on $\partial M$}.
	\end{cases}
	\eeq
	\end{pro}
		\begin{proof}
		The result follows easily from the observations made above, simply noting that \rf{condgo} implies the validity of \rf{condg} for $\overline{g}(x,t)$ defined in \rf{go}, while \rf{condsupo} corresponds to conditions \rf{condsup} for equation \rf{licho}. Thus there exists $v\in \mathrm{C}^2(\inn M)\cap\mathrm{C}^0(M)\cap\mathrm{L}^{\infty}(M)$ positive solution of \rf{licho} from which we obtain $\ds{u=\frac{1}{v}\in \mathrm{C}^2(\inn M)\cap\mathrm{C}^0(M)\cap\mathrm{L}^{\infty}(M)}$ positive solution of \rf{u-}.
		\end{proof}
\section{Proofs of the main results}\label{secbdl}	
Now we put together the results of the previous sections to prove the existence of prositive solutions. 
		\begin{proof}[Proof of Theorem A]
		By \rf{condgall} and \rf{condsupall} it follows that the hypoteses of Proposition \ref{thmsup} and Proposition \ref{thmsub} are satisfied, thus there exist $u^+$, $u^-\in \mathrm{C}^2(\inn M)\cap\mathrm{C}^0(M)\cap\mathrm{L}^{\infty}(M)$ respectively a supersolution of \rf{u+-} and a subsolution of \rf{u-} (and thus also a subsolution of \rf{u+-}). Morever by Remark \ref{stac} we can assume that there exists a $m>0$ such that $u^+\geq m$. Now, for $s\in(0,1)$, set $u_s=su^-$. We claim that $u_s$ is still a subsolution of \rf{u-}, indeed
			\beqs
			\begin{aligned}
			Lu_s-b_{+}(x)u_s^{\sigma}+c(x)u_s^{\tau} & = s \left(Lu^--b_{+}(x)s^{\sigma-1}(u^-)^{\sigma}+c(x)s^{\tau-1}(u^-)^{\tau}\right)\\
			& \geq s\left(Lu^--b_{+}(x)(u^-)^{\sigma}+c(x)(u^-)^{\tau}\right)\\
			& \geq 0\,,
			\end{aligned}
			\eeqs
		on $M$, and since $u_s<u^-$, it follows from the monotonicity of $\ds{\frac{g(x,t)}{t}}$ that
			\beqs
			\begin{aligned}
			\partial_{\nu}u_s-g(x,u_s)& = s \partial_{\nu}u^--g(x,u_s) \\
			&\leq sg(x,u^-)-g(x,u_s)\\
			& =u_s\left(\frac{g(x,u^-)}{u^-}-\frac{g(x,u_s)}{u_s}\right)\\
			&\leq 0,
			\end{aligned}
			\eeqs
		on $\partial M$. Thus $u_s$ is still a subsolution of \rf{u-} for any $s\in(0,1)$. In particular, choosing $\ds{s<\frac{m}{\sup_{x\in M}u^-}}$, we have that $u_s\in \mathrm{C}^2(\inn M)\cap\mathrm{C}^0(M)\cap\mathrm{L}^{\infty}(M)$ is a subsolution of \rf{u+-} such that $0<u^-<u^+$, thus we can apply Proposition \ref{perron} to get the desired positive solution $u\in \mathrm{C}^2(\inn M)\cap\mathrm{C}^0(M)\cap\mathrm{L}^{\infty}(M)$ of \rf{u+-}.
		\end{proof}

The proof of Corollary B is a straightforward application of Theorem A.
		\begin{proof}[Proof of Corollary B]
		The corollary follows from Theorem A once we show that conditions \rf{defoc} implies the validity of \rf{condgall} and the monotonicity of $\ds{\frac{g(x,t)}{t}}$.
		First of all, condition \rf{defoc} i) means that, for any fixed $x\in M$, $\ds{\frac{g(x,t)}{t}}$ is a non-increasing function of $t\in\R^+$, indeed
			\[
			\frac{\partial}{\partial t}\left(\frac{g(x,t)}{t}\right)=\frac{\partial}{\partial t}\sum_{i=1}^Ng_i(x)t^{q_i-1}=\sum_{i=1}^N(q_i-1)g_i(x)t^{q_i-2}\leq0\,.
			\]
		Now, for all $x\in M$
			\[
			\frac{g(x,t)}{t}=g_k(x)t^{q_k-1}+\sum_{i\neq k}g_i(x)t^{q_i-1}\,,
			\]
		but from \rf{defoc} i), ii) it follows that the second summand is bounded above by a positive constant, while
			\[
			\lim_{t\rightarrow+\infty}g_k(x)t^{q_k-1}=-\infty\,,
			\]
		that is, \rf{condgall} iv) is satisfied. Moreover for $\gamma\in\R^+$
			\beqs
			\begin{aligned}
			g(x,\gamma) &=g_k(x)\gamma\left[\sum_{q_i>1}\frac{g_i(x)}{g_k(x)}\gamma^{q_i-1}+\sum_{q_i\leq 1}\frac{g_i(x)}{g_k(x)}\gamma^{q_i-1}\right]\\
			& \leq g_k(x)\gamma\left[\gamma^{q_k-1}-\sum_{q_i\leq 1}\left\|\frac{g_i(x)}{g_k(x)}\right\|_{\mathrm{L}^{\infty}(\partial M)}\gamma^{q_i-1}\right]\\
			\end{aligned}
			\eeqs
		and there exists a $\overline{\gamma}>0$ such that the quantity in square brackets is positive for $\gamma>\overline{\gamma}$, thus \rf{condgall} ii) follows. Conditions \rf{condgall} i), and iii) can be derived similarly.
		\end{proof}

\section{Appendix}
Here we adapt the proof of Theorem 8.38 of \cite{GT} to obtain the following existence result for the Zaremba problem.
	\begin{pro}
	Let $a(x)\in\mathrm{C}^0(M)$ and $\Omega\subset M$ be a non-empty bounded open set such that $\partial\Omega$ is Lipschitz and $\partial_1\Omega$ is at least $\mathrm{C}^2$. Then there exists a unique positive solution $v\in\mathrm{C}^{0,\alpha}(\overline{\Omega})\cap\mathrm{W}^{2,2}(\Omega')$ of
		\beq\label{appeq}
		\begin{cases}
		\Delta v+a(x)v+\eigz(\Omega)v=0 & \hbox{on $\Omega$}\\
		v=0 & \hbox{on $\partial_0\Omega$}\\
		\partial_{\nu}v=0 & \hbox{on $\inn\partial_1\Omega$}\\
		\left\|v\right\|_{\mathrm{L}^2(\Omega)}=1\,,&\,
		\end{cases}
		\eeq
	for some $\alpha>0$ and for all $\Omega'\subset\subset\left(\Omega\cup\partial_1\Omega\right)$, where $\eigz(\Omega)$ is defined as
		\beq\label{appvar}
		\eigz(\Omega)=\inf\left\{\int_{\Omega}\left|\nabla\varphi\right|^2-a(x)\varphi^2\,:\,\varphi\in\mathrm{W}_0^{1,2}(\Omega\cup\partial_1\Omega)\,,\int_{\Omega}\varphi^2=1\right\}\,.
		\eeq
	\end{pro}
	\begin{proof}
	The strategy of the proof consists in showing that a minimizer of \rf{appvar} is a solution of \rf{appeq}. Thus, for $0\neq u\in H:=\mathrm{W}_0^{1,2}(\Omega\cup\partial_1\Omega)$, we define
		\[
		Q(u)=\int_{\Omega}\left|\nabla u\right|^2-a(x)u^2
		\] 
	and the Rayleigh quotient $J(u)$ as
		\[
		J(u)=\frac{Q(u)}{\int_{\Omega} u^2}\,,
		\]
	since $a(x)\in\mathrm{C}^0(M)$, it follows that $J$ is bounded from below and 
		\[
		\eigz(\Omega)=\inf_{u\in H}J(u)\,.
		\]
	Now choose a minimizing sequence $\left\{v_n\right\}\subset H$ such that $\left\|v_n\right\|_{\mathrm{L}^2(\Omega)}=1$. Thus
		\[
		\int_{\Omega}\left|\nabla v_n\right|^2=J(v_n)+\int_{\Omega}a(x)v_n^2\leq J(v_n)+\left\|a\right\|_{\mathrm{L}^{\infty}(\Omega)}<\eigz(\Omega)+1+\left\|a\right\|_{\mathrm{L}^{\infty}(\Omega)}
		\]
	for $n$ big enough, this implies that $\left\{v_n\right\}$ is a bounded sequence in $\mathrm{W}^{1,2}(\Omega)$. Since $\partial\Omega$ is Lipschitz, the embedding $\mathrm{W}^{1,2}(\Omega)\hookrightarrow\mathrm{L}^2(\Omega)$ is compact by the Kondrakov theorem (see for instance Theorem 7.26 of \cite{GT}), thus there exists a subsequence $\left\{v_n\right\}$ which converges in $\mathrm{L}^2(\Omega)$ to a function $v$. A simple computation shows that
		\[
		Q\left(\frac{v_n-v_m}{2}\right)+Q\left(\frac{v_n+v_m}{2}\right)=\frac{1}{2}\left(Q(v_n)+Q(v_m)\right)\,.
		\]
	Thus
		\beqs
		\begin{aligned}
		Q\left(\frac{v_n-v_m}{2}\right)& =\frac{1}{2}\left(Q(v_n)+Q(v_m)\right)-J\left(\frac{v_n+v_m}{2}\right)\left\|\frac{v_n+v_m}{2}\right\|^2_{\mathrm{L}^{2}(\Omega)}\\
		& \leq\frac{1}{2}\left(Q(v_n)+Q(v_m)\right)-\eigz(\Omega)\left\|\frac{v_n+v_m}{2}\right\|^2_{\mathrm{L}^{2}(\Omega)}\,,
		\end{aligned}
		\eeqs
	and the RHS tends to $0$ as $n,m\rightarrow\infty$. This last computation, together with the fact that $\left\{v_n\right\}$ converges in $\mathrm{L}^2(\Omega)$, implies that the RHS of
		\[
		\int_{\Omega}\left|\nabla (v_n-v_m)\right|^2=4Q\left(\frac{v_n-v_m}{2}\right)+\left\|a\right\|_{\mathrm{L}^{\infty}(\Omega)}\left\|v_n-v_m\right\|_{\mathrm{L}^{2}(\Omega)}
		\]
	tends to $0$ as $n,m\rightarrow\infty$, showing that $\left\{v_n\right\}$ is a Cauchy sequence in $\mathrm{W}^{1,2}(\Omega)$. Since $H$ is a closed subspace of $\mathrm{W}^{1,2}(\Omega)$ we can conclude that the limit $v$ is in $H$. Now, a standard argument in calculus of variations shows that		
		\beq\label{EL}
		\int_{\Omega}\g{\nabla v}{\nabla\varphi}-\left(a(x)+\eigz(\Omega)\right)v\varphi=0\,,
		\eeq
	for each $\varphi\in\mathrm{C}^1_0(\Omega\cup\partial_1\Omega)$, that is, $v\in H$ satisfies the Euler-Lagrange equation \rf{appeq} weakly. From Theorem 8.29 of \cite{GT} it follows that $v\in\mathrm{C}^{0,\alpha}(\overline{\Omega})$, while from Theorem 8.8 of \cite{GT} we have that $v\in\mathrm{W}^{2,2}(\Omega')$ for each $\Omega'\subset\subset\Omega$. It remains to show that $v$ is of class $\mathrm{W}^{2,2}$ on open subsets $\Omega'\subset\Omega$ crossing the boundary $\partial_1\Omega$. To show this, consider $x_0\in\inn\partial_1\Omega$, then there exist a neighborhood of $x_0$, $U=U(x_0)$ and a $\mathrm{C}^2$ map $\Psi:\R^m\rightarrow M$ such that it is a diffeomorphism in a neighborhood of $B_1^+(0)$, the upper half-ball of radius $1$, and
		\[
		\Psi(B_1^+(0))=U\,,\quad\quad\Psi(B_1^+(0)\cap\left\{x_n=0\right\})=\partial_1 U\,.
		\]
	The function $w=v\circ\Psi$ belongs to $\mathrm{W}^{1,2}(B_1^+(0))$ and satisfies
		\[
		\int_{B_1^+}A(\nabla w,\nabla\varphi)-\left((a\circ\Psi)(x)+\eigz(\Omega)\right)w\varphi=0\,,
		\]
	for each $\varphi\in\mathrm{C}^1_0(B^+_1)$, where $A(\,,\,)$ is a bounded, symmetric $2$-form. Thus, defining	
		\beqs
		\widetilde{w}(x)=
		\begin{cases}
		w(x_1,x_2,\dots,x_{n-1},x_n)& \hbox{if $x_n\geq0$}\\
		w(x_1,x_2,\dots,x_{n-1},-x_n)& \hbox{if $x_n<0$}\,,\\
		\end{cases}
		\eeqs
	we have that $\widetilde{w}\in\mathrm{W}^{1,2}(B_1(0))$ and the divergence theorem implies that
		\[
		\int_{B_1}A(\nabla \widetilde{w},\nabla\varphi)-\left((a\circ\Psi)+\eigz(\Omega)\right)\widetilde{w}\varphi=0\,,
		\]
	for each $\varphi\in\mathrm{C}^1_0(B_1)$. Another application of Theorem 8.8 of \cite{GT} shows that $\widetilde{w}\in\mathrm{W}^{2,2}(\widetilde{\Omega})$ for each $\widetilde{\Omega}\subset\subset B_1(0)$, in turn, this implies that $v\in\mathrm{W}^{2,2}(\Omega')$ for each $\Omega'\subset\subset\left(\Omega\cup\partial_1\Omega\right)$. Now, using the divergence theorem and the Gagliardo trace theorem (see for instance Theorem 4.12 of \cite{AF}) we can rewrite \rf{EL} as	
		\[
		\int_{\Omega}\varphi\left(\Delta v+\left(a(x)+\eigz(\Omega)\right)v\right)-\int_{\partial_1\Omega}\varphi\partial_{\nu}v=0
		\]
	for each $\varphi\in\mathrm{C}^1_0(\Omega\cup\partial_1\Omega)$. Since $\varphi$ is arbitrary, this last equation implies the validity of \rf{appeq}. The uniqueness of the eigenfunction $v$ is a consequence of the positivity, this is a standard fact and follows from the fact that also $\left|v\right|\in H$ is a minimizer of the Rayleigh quotient and from the Harnack inequality (see for instance Corollary 8.21 of \cite{GT}).
	\end{proof}

\end{document}